\documentclass[12pt]{amsart}
\usepackage{amscd,amsmath,amsthm,amssymb}
\usepackage{pstcol,pst-plot,pst-3d}
\usepackage{lmodern,pst-node}

\def\frk{\frak}               

\def\Phi{{\frk n}}
\def\Phi{{\frk N}}
\def\opn#1#2{\def#1{\operatorname{#2}}} 
\opn\chara{char} \opn\length{\ell} \opn\pd{pd} \opn\rk{rk}
\opn\projdim{proj\,dim} \opn\injdim{inj\,dim} \opn\rank{rank}
\opn\depth{depth} \opn\grade{grade} \opn\height{height}
\opn\embdim{emb\,dim} \opn\codim{codim}

\opn\Tr{Tr} \opn\bigrank{big\,rank}
\opn\superheight{superheight}\opn\lcm{lcm}
\opn\trdeg{tr\,deg}
\opn\reg{reg} \opn\lreg{lreg} \opn\ini{in} \opn\lpd{lpd}
\opn\size{size}\opn\bigsize{bigsize}
\opn\cosize{cosize}\opn\bigcosize{bigcosize}
\opn\sdepth{sdepth}\opn\sreg{sreg}
\opn\link{link}\opn\fdepth{fdepth}
\opn\div{div} \opn\Div{Div} \opn\cl{cl} \opn\Cl{Cl}
\opn\Spec{Spec} \opn\Supp{Supp} \opn\supp{supp} \opn\Sing{Sing}
\opn\Ass{Ass} \opn\Min{Min}\opn\Mon{Mon} \opn\dstab{dstab} \opn\astab{astab}
\opn\Ann{Ann} \opn\Rad{Rad} \opn\Soc{Soc}
\opn\Im{Im} \opn\Ker{Ker} \opn\Coker{Coker} \opn\Am{Am}
\opn\Hom{Hom} \opn\Tor{Tor} \opn\Ext{Ext} \opn\End{End}
\opn\Aut{Aut} \opn\id{id}

\opn\nat{nat}
\opn\pff{pf}
\opn\Pf{Pf} \opn\GL{GL} \opn\SL{SL} \opn\mod{mod} \opn\ord{ord}
\opn\Gin{Gin} \opn\Hilb{Hilb}\opn\sort{sort}
\opn\aff{aff} \opn\con{conv} \opn\relint{relint} \opn\st{st}
\opn\lk{lk} \opn\cn{cn} \opn\core{core} \opn\vol{vol}
\opn\link{link} \opn\star{star}\opn\lex{lex}
\opn\cdeg{cdeg}
\opn\T{T}
\opn\gr{gr}

\def\pot#1#2{#1[\kern-0.28ex[#2]\kern-0.28ex]}

\opn\dirlim{\underrightarrow{\lim}}
\opn\inivlim{\underleftarrow{\lim}}
\let\union=\cup
\let\sect=\cap

\let\to=\rightarrow
\let\To=\longrightarrow
\def\Implies{\ifmmode\Longrightarrow \else
        \unskip${}\Longrightarrow{}$\ignorespaces\fi}
\def\implies{\ifmmode\Rightarrow \else
        \unskip${}\Rightarrow{}$\ignorespaces\fi}
\def\iff{\ifmmode\Longleftrightarrow \else
        \unskip${}\Longleftrightarrow{}$\ignorespaces\fi}

\let\:=\colon
\newtheorem{Theorem}{Theorem}[section]

\newtheorem{Corollary}[Theorem]{Corollary}
\newtheorem{Proposition}[Theorem]{Proposition}

\newtheorem{Definition}[Theorem]{Definition}

\let\epsilon\varepsilon
\let\phi=\varphi
\let\kappa=\varkappa
\def\qed{\ifhmode\textqed\fi
      \ifmmode\ifinner\quad\qedsymbol\else\dispqed\fi\fi}
\def\textqed{\unskip\nobreak\penalty50
       \hskip2em\hbox{}\nobreak\hfil\qedsymbol
       \parfillskip=0pt \finalhyphendemerits=0}
\def\dispqed{\rlap{\qquad\qedsymbol}}

\opn\dis{dis}
\def\pnt{{\raise0.5mm\hbox{\large\bf.}}}

\opn\Lex{Lex}

\author{J\"Urgen Herzog}

\address{Fachbereich Mathematik, Universit\"at Duisburg-Essen, Campus Essen, 45117
Essen, Germany}
\email{juergen.herzog@uni-essen.de}

\author{Giancarlo Rinaldo}
\address{Department of Mathematics,
University of Trento,
via Sommarive, 14,
38123 Povo (Trento), Italy}
\email{giancarlo.rinaldo@unitn.it}

\subjclass[2010]{Primary 13D02, 13C13; Secondary 05E40, 05C69.}
\keywords{Extremal Betti numbers, regularity,  binomial edge ideals, block graphs}

\begin{document}
\title{On the extremal Betti numbers of binomial edge ideals of block graphs}
\maketitle
\begin{abstract}
We compute one of the distinguished extremal Betti number of the  binomial edge ideal of a block graph, and classify all block graphs admitting precisely one extremal Betti number.
\end{abstract}

\section*{Introduction}

Let $K$ be a field and $I$  a graded ideal in the polynomial ring $S=K[x_1,\ldots,x_n]$. The  most important invariants of $I$,  which are provided by its graded finite free resolution,  are the regularity and the projective dimension of $I$.  In general these invariants are hard to compute. One strategy to bound them is to consider for  some monomial order the initial ideal $\ini_< (I)$ of $I$. It is known that for the graded Betti numbers one has $\beta_{i,j}(I)\leq \beta_{i,j}(\ini_<(I))$. This fact implies in particular that $\reg(I)\leq \reg(\ini_<(I))$, and $\projdim (I)\leq \projdim(\ini_<(I))$. In general however these inequalities may be strict. On the other hand,  it is known that if $I$ is the defining  binomial ideal of a toric  ring, then $\projdim (I)= \projdim \ini_<(I)$, provided $\ini_<(I)$ is a squarefree monomial ideal. This is a consequence of a theorem of Sturmfels \cite{St}. The first author of this paper conjectures that whenever the initial ideal of a graded ideal $I\subset S$ is a squarefree monomial ideal, then the extremal Betti numbers of $I$ and $\ini_<(I)$ coincide in their positions and values. This conjecture implies that $\reg(I)=\reg (\ini_<(I))$ and $\projdim (I)=\projdim(\ini_<(I))$  for any ideal $I$ whose initial ideal is squarefree.

An interesting class of binomial ideals having the property that all of its initial ideals are squarefree monomial ideals are the so-called binomial edge ideals, see \cite{HH}, \cite{CR}, \cite{BBS}. Thus it is  natural to test the above conjectures for binomial edge ideals. A positive answer to this conjecture was given in \cite{HEH} for Cohen-Macaulay binomial edge ideals of PI graphs (proper interval graphs). In that case actually  all the graded Betti numbers of the binomial edge ideal and its initial ideal coincide. It is an open question wether this happens to be true for any  binomial edge ideal of a PI graph. Recently this has been confirmed to be true, if the PI graph consists of at most two cliques \cite{Ba}.  In general the graded Betti numbers are known only for very special classes of graphs including cycles \cite{ZZ}.

Let $J_G$ denote the binomial edge ideal of a graph $G$. The first result showing that  $\reg(J_G)=\reg(\ini_<(J_G))$ without computing all graded Betti numbers was obtained for PI graphs by Ene and Zarojanu \cite{EZ}. Later Chaudhry, Dokuyucu and Irfan \cite{CDI} showed that $\projdim (J_G)=\projdim (\ini_<(J_G))$ for any block graph $G$, and $\reg(J_G)=\reg (\ini_<(J_G))$ for a special class of block graphs. Roughly speaking, block graphs are trees whose edges are replaced by cliques. The blocks of a graph are the biconnected components of the graph, which for a block graph are all cliques. In particular trees are block graphs. It is still an open problem to determine the regularity of the binomial edge ideal for block graphs (and even for trees) in terms of the combinatorics of the graph. However, strong lower and upper bounds for the regularity of edge ideals are known by  Matsuda and Murai \cite{MM} and Kiani and Saeedi Madani \cite{KM}. Furthermore, Kiani and Saeedi Madani  characterized all graphs are whose binomial edge ideal have regularity 2 and regularity 3, see  \cite{MK} and \cite{MK1}.

In this note we determine the position and value of one of the distinguished extremal Betti number of the binomial edge ideal of a block graph. Let $M$ be a finitely graded $S$-module. Recall that a  graded  Betti number $\beta_{i,i+j}( M)\neq 0$ of $M$  is called an {\em extremal}, if  $\beta_{k,k+l}(M)=0$ for all pairs $(k,l)\neq (i,j)$ with $k\geq i$ and $l\geq j$. Let $q=\reg(M)$ and $p=\projdim(M)$, then  there exist unique numbers $i$ and $j$ such that
$\beta_{i,i+q}(M)$ and $\beta_{p,p+j}(M)$ are extremal  Betti numbers.  We call them the {\em distinguished extremal Betti numbers}  of $M$. The distinguished extremal Betti numbers are different from each other if and only if $M$ has more than two extremal Betti numbers.

In order to describe our result in detail, we introduce the following concepts.  Let $G$ be finite simple graph. Let $V(G)$ be the vertex set  and $E(G)$ the edge set of $G$. The clique degree of a vertex $v\in V(G)$, denoted $\cdeg(v)$,  is the number of cliques to which it belongs. For a tree the clique degree of a vertex is just the ordinary degree. A vertex $v\in V(G)$ is a called a free vertex, if $\cdeg(v)=1$ and an inner vertex if $\cdeg(v)>1$. Suppose $v\in V(G)$ is a vertex of clique degree $2$. Then $G$ can be decomposed as a union of subgraphs $G_1\union G_2$ with $V(G_1)\sect V(G_2)=\{v\}$ and where $v$ is a free vertex of $G_1$ and $G_2$. If this is the case, we say that $G$ is {\em decomposable}. In Proposition~\ref{decomposable} we show that if $G$ decomposable with $G=G_1\union G_2$, then the graded Poincar\'{e} series of $G$ is just the product of the graded Poincar\'{e} series of $S/J_{G_1}$ and $S/J_{G_2}$. This result, with a simplified proof, generalizes a theorem of the second author  which he obtained in a joint paper with Rauf \cite{RR}. As a consequence one obtains that the position and value of the distinguished extremal Betti numbers of $S/J_G$ are obtained  by adding the positions and multiplying the values of the corresponding distinguished extremal Betti numbers of $S/J_{G_1}$ and $S/J_{G_2}$. The other extremal Betti numbers of $S/J_G$ are not obtained in this simple way from those of $S/J_{G_1}$ and $S/J_{G_2}$. But the result shows that if we want to determine the distinguished extremal Betti numbers of $S/J_G$ for  a graph $G$ (which also give us the regularity and projective dimension of  $S/J_G$), it suffices to assume that $G$ is indecomposable.

Let $f(G)$ be the number of free vertices and $i(G)$ the number of inner vertices of $G$. In Theorem~\ref{main} we show: let $G$ be an indecomposable block graph with $n$ vertices. Furthermore let $<$ be the lexicographic order induced by $x_1>x_2>\ldots > x_n>y_1>y_2>\cdots >y_n$. Then $\beta_{n-1,n-1+i(G)+1}(S/J_G)$ and $\beta_{n-1,n-1+i(G)+1}(S/\ini_<(J_G))$ are extremal Betti numbers, and
$
 \beta_{n-1,n-1+i(G)+1}(S/\ini_<(J_G)) =\beta_{n-1,n-1+i(G)+1}(S/J_G)=f(G)-1.
$
The theorem implies that $\reg(J_G)\geq i(G)$. It also implies that  $\reg(J_G)=i(G)$ if and only if $S/J_G$ has exactly one extremal Betti number, namely the Betti number $\beta_{n-1,n-1+i(G)+1}(S/J_G)$. In Theorem~\ref{hope} we classify all block graphs with the property that they admit  precisely one extremal Betti number, by listing the  forbidden induced subgraphs (which are $4$ in total), and we  also give an explicit description of the  block graphs with precisely one extremal Betti number.  Carla Mascia informed us that Jananthan et al  in an yet unpublished paper and revised version of \cite{JNR} obtained a related  result for trees.

For indecomposable block graphs $G$  the most challenging  open  problem  is to obtain a combinatorial formula for the regularity of $J_G$, or even better,  a description of {\em both}  distinguished extremal Betti numbers of $S/J_G$.

 \section{Decomposable  graphs and binomial edge ideals}

Let $G$ be a graph with vertex set $V(G)=[n]$ and edge set $E(G)$. Throughout this paper, unless otherwise stated, we will assume that $G$ is connected.

A subset $C$ of $V(G)$ is called a \textit{clique} of $G$ if for all $i$ and $j$ belonging to $C$ with $i \neq j$ one has $\{i, j\} \in E(G)$.

\begin{Definition}{\em  Let $G$ be a graph and $v$ a vertex of $G$. The  \textit{clique degree} of  $v$, denoted $\cdeg{v}$,   is  the number of maximal cliques to which $v$ belongs. }
 \end{Definition}

A vertex $v$ of $G$ is called a {\em free vertex} of $G$, if $\cdeg(v)=1$, and  is called an {\em inner vertex}, if $\cdeg(v)>1$. We denote by $f(G)$ the number of free vertices of $G$ and by $i(G)$ the number of inner vertices of $G$.

\begin{Definition}\label{Def:indecomposable}{\em
  A graph $G$ is {\em decomposable},   if there exist two subgraphs $G_1$ and $G_2$ of $G$,  and a decomposition
  \begin{equation}\label{eq:dec2}
  G=G_1\cup G_2
  \end{equation}
  with $\{v\}=V(G_1)\cap V(G_2)$, where   $v$ is a free vertex of $G_1$ and $G_2$.

  If $G$ is  not decomposable, we call it {\em indecomposable}.}
\end{Definition}

Note that any graph has a unique decomposition (up to ordering)
\begin{equation}\label{eq:decr}
     G=G_1\cup \cdots \cup G_r,
  \end{equation}
  where $G_1,\ldots,G_r$ are indecomposable subgraphs  of $G$,  and for $1\leq i<j\leq r$ either $V(G_i)\cap V(G_j) = \emptyset$ or $V(G_i)\cap V(G_j) =\{v\}$ and  $v$ is a free vertex of $G_i$ and $G_j$.

  \medskip
  For a graded $S$-module $M$ we denote by $B_M(s,t)=\sum_{i,j}\beta_{ij}(M)s^it^j$ the Betti polynomial of $M$. The following proposition generalizes a result  due to Rinaldo and Rauf \cite{RR}.

\begin{Proposition}\label{decomposable}
Let $G$ be a decomposable graph, and let  $G=G_1\cup G_2$  be a  decomposition of $G$. Then
\[
B_{S/J_G}(s,t)= B_{S/J_{G_1}}(s,t)B_{S/J_{G_2}}(s,t).
\]
\end{Proposition}

\begin{proof}
 We may assume that  $V(G)=[n]$ and $V(G_1)=[1,m]$ and $V(G_2)=[m, n]$ . We claim that
 for the lexicographic order $<$  induced by $x_1>x_2>\cdots >x_n>y_1>y_2>\cdots >y_n$,  we have
 \[
 \ini_<(J_{G_1})\subset K[\{x_i,y_i\}_{i=1,\ldots, m-1}][y_m]\quad \text{and}\quad  \ini_<(J_{G_1})\subset K[\{x_i,y_i\}_{i=m+1,\ldots, n}][x_m].
 \]
 We recall the notion of admissible path, introduced in  \cite{HH} in order to compute Gr\"obner bases of binomial edge ideals. A path $\pi:i=i_0,i_1,\ldots,i_r=j$ in a graph $G$ is called {\em admissible}, if
\begin{enumerate}
 \item $i_k\neq i_\ell$ for $k\neq \ell$;
 \item for each $k=1,\ldots,r-1$ one has either $i_k<i$ or $i_k>j$;
 \item for any proper subset $\{j_1,\ldots,j_s\}$ of $\{i_1,\ldots,i_{r-1}\}$, the sequence $i,j_1,\ldots,j_s,j$ is not a path.
\end{enumerate}
Given an admissible path $\pi:i=i_0,i_1,\ldots,i_r=j$ from $i$ to $j$ with $i<j$ we associate the monomial  $u_\pi=(\prod_{i_k>j}x_{i_k})(\prod_{i_\ell<i}y_{i_\ell})$. In \cite{HH} it is shown that
\[
\ini_< (J_G)=(x_iy_ju_\pi\:\; \pi \mbox{ is an admissible path}).
\]
The claim follows by observing that the only admissible paths passing through the vertex $m$ are the ones inducing the set  of monomials
\[
\{x_iy_m u_\pi\:\; V(\pi)\in V(G_1)\}\union\{ x_my_j u_\pi \:\; V(\pi)\in V(G_2)\}.
\]
We have the following cases to study
\begin{enumerate}
 \item[(a)] $V(\pi)\subset V(G_1)$ or $V(\pi)\subset V(G_2)$;
 \item[(b)] $V(\pi)\cap V(G_1)\neq \emptyset$ and $V(\pi)\cap V(G_2)\neq \emptyset$.
\end{enumerate}

(a) We may assume that $V(\pi)\subset V(G_1)$. Assume $m$ is not an endpoint of $\pi$. Then  $\pi:i=i_0,\ldots,i_r=j$ with $m=i_k$, $0<k<r$. Since $i_{k-1}$ and $i_{k+1}$ belong to the maximal clique in $G_1$ containing $m$, it follows that $\{i_{k-1},i_{k+1}\}\in E(G_1)$ and condition (3) is not satisfied. Therefore $\pi:i=i_0,\ldots,i_r=m$ and
$x_iy_m u_\pi$ is the corresponding monomial.

(b) In this case we observe that $m$ is not an endpoint of the path $\pi:i=i_0,\ldots,i_r=j$. Since $i<m<j$ this path is not admissible by (2).
\medskip

Now the claim implies that $\Tor_i(S/\ini_<(J_{G_1}), S/\ini_<(J_{G_2}))=0$ for $i>0$. Therefore,  we also have $\Tor_i(S/J_{G_1}, S/J_{G_2})=0$  for $i>0$. This yields the desired conclusion.
\end{proof}

The proposition implies that $\projdim S/J_G=\projdim S/J_{G_1}+\projdim S/J_{G_2}$ and $\reg S/J_G=\reg S/J_{G_1}+\reg S/J_{G_2}$. In fact, much more is true.
Let $M$ be a finitely graded $S$-module. A Betti number $\beta_{i,i+j}( M)\neq 0$ is called an {\em extremal}  Betti number of $M$, if  $\beta_{k,k+l}(M)=0$ for all pairs $(k,l)\neq (i,j)$ with $k\geq i$ and $l\geq j$. Let $q=\reg(M)$ and $p=\projdim(M)$, then  there exist unique numbers $i$ and $j$ such that
$\beta_{i,i+q}(M)$ and $\beta_{p,p+j}(M)$ are extremal  Betti numbers. We call them the {\em distinguished extremal Betti numbers}  of $M$. $M$ admits only one extremal Betti number if and only the  two distinguished extremal Betti numbers are equal.

\begin{Corollary}
With the assumptions of Proposition~\ref{decomposable},  let $\{\beta_{i_t,i_t+j_t}(S/J_{G_1})\}_{t=1,\ldots, r}$ be the set of extremal Betti numbers of $S/J_{G_1}$ and $\{\beta_{k_t,k_t+l_t}(S/J_{G_2})\}_{t=1,\ldots, s}$ be the set of extremal Betti numbers of $S/J_{G_2}$. Then
$
\{\beta_{i_t+k_{t'}, (i_t+k_{t'})+(j_t+l_{t'})}(S/J_G)\}_{t=1,\ldots r \atop t'=1,\ldots,s}
$
is a  subset of the extremal Betti numbers of $S/J_G$.

For $k=1,2$, let $\beta_{i_k,i_k+q_k}(G_k)$ and $\beta_{p_k,p_k+j_k}(G_k)$  be the  distinguished extremal Betti numbers of $G_1$ and $G_2$. Then
 $\beta_{i_1+i_2,i_1+i_2+q_1+q_2}(G)$ and $\beta_{p_1+p_2,p_1+p_2+j_1+j_2}(G)$ are the  distinguished extremal Betti numbers of $G$,  and
\begin{eqnarray*}
\beta_{i_1+i_2,i_1+i_2+q_1+q_2}(G)&=&\beta_{i_1,i_1+q_1}(G_1)\beta_{i_2,i_2+q_2}(G_2),\\
\beta_{p_1+p_2,p_1+p_2+j_1+j_2}(G)&= &\beta_{p_1,p_1+j_1}(G_1)\beta_{p_2,p_2+j_2}(G_2).
\end{eqnarray*}
\end{Corollary}

%

\section{Extremal Betti numbers of Block graphs}

Let $G$ be a graph with vertex set $V(G)$ and edge set $E(G)$. A vertex of a graph is called a \textit{cutpoint} if the removal of the vertex increases the number of connected components. A connected subgraph of $G$ that has no cutpoint and is maximal with respect to this property is called a \textit{block}.

\begin{Definition}{\em  A graph $G$ is called a {\em block graph}, if each block of $G$ is a clique. }
 \end{Definition}

Observe that  a block graph $G$ is decomposable  if and only if there exists $v\in V(G)$ with $\cdeg(v)=2$. In particular, a block graph is indecomposable, if $\cdeg(v)\neq 2$ for all $v\in V(G)$.

 A block $C$ of the block graph $G$ is called a {\em leaf} of $G$,  if there is exactly one $v\in V(C)$ with $\cdeg(v)>1$.

\begin{Theorem}
\label{main}
Let $G$ be an indecomposable  block graph with $n$ vertices. Furthermore let $<$ be the lexicographic order induced by $x_1>x_2>\ldots > x_n>y_1>y_2>\cdots >y_n$. Then $\beta_{n-1,n-1+i(G)+1}(S/J_G)$ and $\beta_{n-1,n-1+i(G)+1}(S/\ini_<(J_G))$ are extremal Betti numbers  of $S/J_G$ and $S/\ini_<(J_G)$, respectively. Moreover,
\[
 \beta_{n-1,n-1+i(G)+1}(S/\ini_<(J_G)) =\beta_{n-1,n-1+i(G)+1}(S/J_G)=f(G)-1.
\]
\end{Theorem}

\begin{proof}
We prove the theorem by induction on $i(G)$. If $i(G)=0$, then $G$ is a clique and $J_G$ is the ideal of $2$-minors of the $2\times n$ matrix
\begin{equation}\label{matrix}
\begin{pmatrix}
x_1 & x_2 & \ldots & x_n \\
y_1 & y_2 & \ldots & y_n
\end{pmatrix}
\end{equation}
The desired conclusion follows by the Eagon-Northcott resolution \cite{EN}.
Let us now assume that the above equation holds for $i(G)>0$.

Let $C_1,\ldots, C_t$ be the blocks of $G$ and assume that $C_t$ is a leaf of $G$. Since $i(G)>0$,  it follows that $t>1$. Let $i$ be the vertex of $C_t$ of $\cdeg(i)>1$, and let $G'$ be the graph which  is obtained from $G$ by replacing  $C_t$ by the clique  whose   vertex set  is the union of the vertices  of the $C_i$ which have a non-trivial intersection with $C_t$. Furthermore,  let $G''$ be the graph which  is obtained from $G$ by removing the vertex $i$, and $H$ be the graph obtained by removing the vertex $i$ from $G'$.

Note that $G'$   and $H$ are indecomposable block graphs for which  $i(G')=i(H)=i(G)-1$.

The following exact sequence
\begin{equation}\label{Exact}
 0\To S/J_G \To S/J_{G'}\oplus S/((x_i,y_i)+J_{G''})\To S/((x_i,y_i)+J_H)
\end{equation}
from \cite{HEH} is used for our induction step. By the proof of  \cite[Theorem 1.1]{HEH} we know that $\projdim S/J_G=\projdim S/J_{G'}=n-1$, $\projdim S/((x_i,y_i)+J_H)=n$,  and $\projdim S/((x_i,y_i)+J_{G''})=n-q$, where $q+1$ is the number of connected components of $G''$. Since $\cdeg(i)\geq 3$ it follows that $q\geq 2$. Therefore,  $\Tor_{n-1}(S/((x_i,y_i)+J_{G''}),K)=0$, and hence  for each $j$, the exact sequence \eqref{Exact} yields the long exact sequence
\begin{equation}\label{LongTor}
 0\to \T_{n,n+j-1}(S/((x_i,y_i)+J_H))\to \T_{n-1,n-1+j}(S/J_{G}) \to \T_{n-1,n-1+j}(S/J_{G'})\to
\end{equation}
where for any finitely generated graded $S$-module, $\T^S_{k,l}(M)$ stands for $\Tor^S_{k,l}(M,K)$.
Note that
\begin{equation}\label{TorIsomorphism}
 \T^S_{n,n+j-1}(S/((x_i,y_i)+J_H)) \cong  \T^{S'}_{n-2,n-2+(j-1)}(S'/J_H),
\end{equation}
where $S'=S/(x_i,y_i)$.

Our induction hypothesis implies that
\[
\T_{n-2,n-2+(j-1)}(S'/J_H)=0\quad \text{for}\quad j>i(H)+2=i(G)+1,
\]
and
\[
\T_{n-1,n-1+j}(S/J_{G'})=0\quad \text{for}\quad j>i(G')+1=i(G).
\]
Now \eqref{LongTor} and \eqref{TorIsomorphism} imply that  $\T_{n-1,n-1+j}(S/J_{G})=0$ for $j>i(G)+1$, and
\begin{equation}\label{TorIsomorphisII}
 \T^{S'}_{n-2,n-2+i(H)+1}(S'/J_H)\cong \T^S_{n-1,n-1+(i(G)+1)}(S/J_{G}).
\end{equation}
By induction hypothesis, $\beta^{S'}_{n-2,n-2+i(H)+1}(S'/J_H)=f(H)-1$. Since $f(G)=f(H)$,   \eqref{TorIsomorphisII} implies that $\beta_{n-1,n-1+i(G)+1}(S/J_G)=f(G)-1$, and together with \eqref{TorIsomorphism} it follows that $\beta_{n-1,n-1+i(G)+1}(S/J_G)$ is an extremal Betti number.

Now we prove the assertions regarding  $\ini_<(J_G)$. If $i(G)=0$, then $J_G$ is the ideal of $2$-minors of the matrix \eqref{matrix}. It is known that  $\ini_<(J_G)$ has a $2$-linear resolution. This implies that $\beta_{i,j}(J_G)=\beta_{i,j}(\ini_<(J_G))$, see \cite{}. This proves the assertions for $i(G)=0$.

Next assume that  $i(G)>0$. As noted in \cite{CDI},  one also has the exact sequence
\[\label{ExactIni}
 0\To S/\ini_< (J_G) \To S/\ini_<(J_{G'})\oplus S/\ini_<((x_i,y_i)+J_{G''})\To S/\ini_<((x_i,y_i)+J_H).
\]

Since $\ini_<((x_i,y_i)+J_H)=(x_i,y_i)+\ini_<(J_H)$ it follows  that
\begin{equation}\label{TorIsomorphismIni}
 \T^S_{n,n+j-1}(S/\ini_<((x_i,y_i)+J_H)) \cong  \T^{S'}_{n-2,n-2+(j-1)}(S'/\ini_<(J_H)).
\end{equation}
Therefore, by using the induction hypothesis, one deduces as before that
\begin{equation}\label{TorIsomorphisIIIni}
 \T^{S'}_{n-2,n-2+i(H)+1}(S'/\ini_<(J_H))\cong \T^S_{n-1,n-1+i(G)+1}(S/\ini_<(J_{G})).
\end{equation}
This concludes the proof.
\end{proof}
The next corollary is an immediate consequence of Proposition \ref{decomposable} and Theorem~\ref{main}.

\begin{Corollary}\label{Lem:Extremal}
 Let $G$ be a  block graph for which $G=G_1\union \cdots \union G_s$ is the decomposition of $G$ into indecomposable graphs. Then each $G_i$ is a block graph, $\beta_{n-1,n-1+i(G)+s}(S/J_G)$ is an extremal Betti number of $S/J_G$ and
\[
 \beta_{n-1,n-1+i(G)+s}(S/J_G)=\prod_{i=1}^s (f(G_i)-1).
\]
\end{Corollary}

In the following theorem we classify all block graphs which admit precisely one extremal Betti number.

\begin{Theorem}\label{hope}
 Let $G$ be a  indecomposable block  graph. Then
 \begin{enumerate}
  \item[(a)] $\reg(S/J_G)\geq i(G)+1$.
  \item[(b)] The following conditions are equivalent:
  \begin{enumerate}
   \item[(i)] $S/J_G$ admits precisely one extremal Betti number.
   \item[(ii)] $G$ does not contain one of the induced subgraphs  $T_0$, $T_1$, $T_2$, $T_3$ of Fig.\ref{avoid}.

   \item[(iii)] Let $P=\{v\in V(G)\:\; \deg(v)\neq 1\}.$ Then each cut point of $G_{|P}$ belongs to exactly two maximal cliques.
  \end{enumerate}
 \end{enumerate}
\end{Theorem}

\begin{proof}

(a) is an immediate consequence of Theorem \ref{main}.

 (b)(i)\implies (ii): Suppose that $G$ contains  one of the induced subgraphs  $T_0$, $T_1$, $T_2$, $T_3$. We will show that $\reg(S/J_G)>i(G)+1$. By Corollary~\ref{Lem:Extremal} this is equivalent to saying that $S/J_G$ admits at least two extremal Betti numbers. To proceed in our proof we shall need the following result \cite[Corollay 2.2]{MM} of Matsuda and Murai which says that for $W\subset V(G)$, one has  $\beta_{ij}(J_{G_{|W|}}) \leq \beta_{ij}(J_G))$ for all $i$ and $j$.

It can be checked by CoCoA that $\reg(S/J_{T_j})>i(T_j)+1$ for each $T_j$. Now assume that $G$ properly contains one of the $T_j$ as induced subgraph. Since $G$ is connected, there exists a clique $C$ of $G$ and subgraph $G'$ of $G$ such that (1) $G'$ contains  one of the $T_j$ as induced subgraph, (2) $V(G')\sect V(C)=\{v\}$. By using induction on the number of cliques of $G$, we may assume that $\reg(S/J_{G'})>i(G')+1$. If $\cdeg(v)=2$, then $i(G)=i(G')+1$ and $\reg(S/J_{G})= \reg(S/J_{G'})+1$, by Proposition~\ref{decomposable}. If $\cdeg(v)>2$, then  $i(G)=i(G')$, and by Matsuda and Murai we have $\reg S/J_{G'}\leq \reg S/J_{G}$. Thus in both case we obtain $\reg(S/J_G)>i(G)$, as desired.

 (ii)\implies(iii): Suppose condition (iii) is not satisfied. Let $C_1,\ldots, C_r$ with $r\geq 3$ be maximal cliques of $G_{|P}$ that meet in the same cutpoint $i$. After a suitable relabeling of the cliques $C_i$ we may assume that one  of the following cases occurs:
 \begin{enumerate}
  \item [($\alpha$)] $C_1,C_2,C_3$ have cardinality $\geq 3$;
  \item [($\beta$)] $C_1,C_2$ have cardinality $\geq 3$, the others have cardinality $2$;
  \item [($\gamma$)] $C_1$ has cardinality $\geq 3$, the others have cardinality $2$;
  \item [($\delta$)] $C_1,\ldots, C_r$ have cardinality $2$.
 \end{enumerate}
 In case ($\alpha$) observe  that $G$ contains $C_1$, $C_2$ and $C_3$, too. But this contradicts the fact that $G$ does not contain $T_0$ as an induced subgraph.
 Similarly in case ($\beta$),  $G$ contains $C_1$, $C_2$. Let $C_3=\{i,j\}$. Since $C_3$ is an edge in $G_{|P}$, it cannot be a leaf of $G$. Therefore, since $G$ is indecomposable, there exist at least two maximal cliques in $G$ for which $j$ is a cut point. It follows that $T_1$ is an induced subgraph, a contradiction.  ($\gamma$) and ($\delta$) are discussed in a similar way.

(iii)\implies (i):   We use induction on $i(G)$. If $i(G)=0$, then $G$ is a clique and the assertion is obvious.  Now let us assume that $i(G)>0$.
By (a)  it is sufficient to prove that $\reg S/J_G\leq i(G)+1$. If $i(G)=0$, then $G$ is a clique and the assertion is obvious. We choose a leaf of $G$. Let $j$ be the unique cut point of this leaf,   and let $G'$, $G''$ and $H$  be the subgraphs of $G$, as defined with respect to $j$ in the proof of  Theorem~\ref{main}.  Note the $G'$ and $H$  are block graphs satisfying the  conditions in (iii) with $i(G')=i(H)=i(G)-1$.  By our induction hypothesis, we have $\reg(S/J_{G'})=\reg(S/J_H)=i(G)$. The graph $G''$ has $\cdeg(j)$ many connected components with one components $G_0$ satisfying (iii) and $i(G_0)=i(G)-1$, and with  the other components being  cliques,  where all,  but possible one of the cliques, say $C_0$, are isolated vertices. Applying our induction hypothesis we obtain that $\reg(S/J_{G''})=i(G_0)+\reg(J_{C_0})\leq i(G)-1+\reg(S/J_{C_0})\leq i(G)$, since $\reg(S/J_{C_0})\leq 1$.

Thus the exact sequence (\ref{Exact}) yields $$\reg S/J_G\leq \max\{\reg S/J_{G'}, \reg S/J_{G''}, \reg S/J_{H}+1\}=i(G)+1,$$ as desired.
\end{proof}

\begin{figure}
\begin{center}
\setlength{\unitlength}{.35cm}
\begin{picture}(38,7)
\newsavebox{\Tri}

\savebox{\Tri}
  (04,03)[bl]{
  \put(00,00){\circle*{.3}}
  \put(04,00){\circle*{.3}}
  \put(02,03){\circle*{.3}}

  \put(00,00){\line(2,3){2}}
  \put(00,00){\line(1,0){4}}
  \put(02,03){\line(2,-3){2}}
}

\put(03.6,06.8){$T_0$}
\put(13.6,06.8){$T_1$}
\put(23.6,06.8){$T_2$}
\put(33.6,06.8){$T_3$}

\put(00,03){\usebox{\Tri}}
\put(04,03){\usebox{\Tri}}
\put(02,00){\usebox{\Tri}}

\put(10,03){\usebox{\Tri}}
\put(14,03){\usebox{\Tri}}

\put(14,03){\line(0,-3){3}}
\put(11,00){\line(1,0){6}}

\put(11,00){\circle*{.3}}
\put(14,00){\circle*{.3}}
\put(17,00){\circle*{.3}}

\put(20,00){\line(1,3){1}}
\put(20,06){\line(1,-3){1}}
\put(21,03){\line(1,0){6}}
\put(28,00){\line(-1,3){1}}
\put(28,06){\line(-1,-3){1}}
\put(24,03){\line(2,3){2}}
\put(24,03){\line(-2,3){2}}
\put(22,06){\line(1,0){4}}

\put(20,00){\circle*{.3}}
\put(21,03){\circle*{.3}}
\put(20,06){\circle*{.3}}
\put(24,03){\circle*{.3}}
\put(27,03){\circle*{.3}}
\put(28,00){\circle*{.3}}
\put(28,06){\circle*{.3}}

\put(22,06){\circle*{.3}}
\put(26,06){\circle*{.3}}

\put(30,00){\line(1,3){1}}
\put(30,06){\line(1,-3){1}}
\put(31,03){\line(1,0){6}}
\put(38,00){\line(-1,3){1}}
\put(38,06){\line(-1,-3){1}}
\put(34,03){\line(0,1){3}}
\put(31,06){\line(1,0){6}}

\put(30,00){\circle*{.3}}
\put(31,03){\circle*{.3}}
\put(30,06){\circle*{.3}}
\put(34,03){\circle*{.3}}
\put(37,03){\circle*{.3}}
\put(38,00){\circle*{.3}}
\put(38,06){\circle*{.3}}

\put(31,06){\circle*{.3}}
\put(34,06){\circle*{.3}}
\put(37,06){\circle*{.3}}

\end{picture}

\end{center}
\caption{Induced subgraphs to avoid}\label{avoid}
\end{figure}
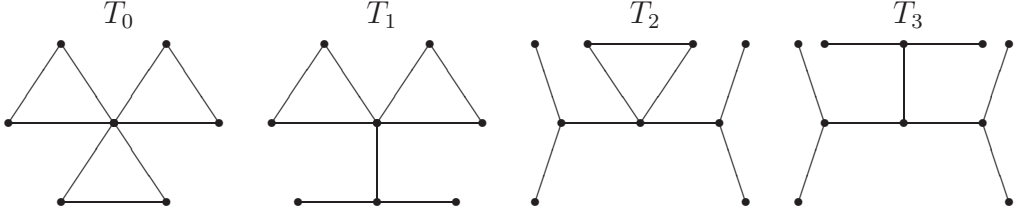

\end{document}